\documentclass[11pt, a4paper]{scrartcl}
\usepackage[english]{babel}

\usepackage{color,enumitem,graphicx, caption}
\usepackage[colorlinks=true,urlcolor=blue,%
citecolor=red,linkcolor=blue,linktocpage,bookmarksnumbered,bookmarksopen]{hyperref}

\usepackage[normalem]{ulem}
\usepackage{amsfonts,amsmath,amssymb,latexsym,soul,amsthm, tikz-cd, braket, amsthm, mathtools, graphicx,array,color,colortbl,xcolor}
\usepackage{diagbox}
\usepackage{todonotes} 

\DeclarePairedDelimiter{\abs}{\lvert}{\rvert}
\DeclarePairedDelimiter{\norm}{\lVert}{\rVert}
\usepackage[T1]{fontenc}
\usepackage[utf8]{inputenc}
\usepackage{csquotes}
\usepackage{lmodern}
\usepackage{microtype}
\usepackage{mathrsfs} 

\usepackage[numbers]{natbib}
\newcommand{\dz}{{\mathrm{d}z}}
\newcommand{\dt}{{\mathrm{d}t}}

\usepackage{fixmath}

\theoremstyle{plain}
\newtheorem{thm}{Theorem}[section]
\newtheorem{lemma}[thm]{Lemma}

\newtheorem{prop}[thm]{Proposition}
\newtheorem{cor}[thm]{Corollary}

\theoremstyle{remark}	
\newtheorem{rem}[thm]{Remark}

\theoremstyle{definition}	
\newtheorem{definizione}[thm]{Definition}

\begin{document}
	\title{Mountain Pass Solutions for an entire semipositone problem involving the Grushin Subelliptic Operator}
	\author{Giovanni Molica Bisci\thanks{Partially supported by GNAMPA 2024: "Aspetti geometrici e analitici di alcuni problemi locali e non-locali in mancanza di compattezza"} \and Paolo Malanchini \and Simone Secchi\thanks{Partially supported by GNAMPA 2024: "Aspetti geometrici e analitici di alcuni problemi locali e non-locali in mancanza di compattezza"}}
	\maketitle
	\setcounter{secnumdepth}{2} 
	\setcounter{tocdepth}{2}
	
\begin{abstract}
\noindent \textbf{Abstract.} For $N\ge 3$ we study the following semipositone problem
\begin{equation*}
	-\Delta_\gamma u = g(z) f_a(u) \quad \hbox{in $\mathbb{R}^N$}, 
\end{equation*}
where $\Delta_\gamma$ is the Grushin operator
\begin{displaymath}
	\Delta_	\gamma u(z) = \Delta_x u(z) + \abs{x}^{2\gamma} \Delta_y u (z) \quad (\gamma\ge 0),
\end{displaymath}
$g\in L^1(\mathbb{R}^N)\cap L^\infty(\mathbb{R}^N)$ is a positive function,  $a>0$ is a parameter and $f_a$ is a continuous function on $\mathbb{R}$ that coincides with $f(t) -a$ for $t\in\mathbb{R}^+$, where $f$ is a continuous function with subcritical and Ambrosetti-Rabinowitz type growth and which satisfies $f(0) = 0$.  Depending on the range of $a$, we obtain the existence of positive mountain pass solutions in $D_\gamma(\mathbb{R}^N)$, extending the results found in  \cite{alves,biswas2023study, biswas2024semipositone}.
\end{abstract}

\section{Introduction}

The term \emph{semipositone} is usually associated to a semilinear elliptic equation of the form
\begin{displaymath}
    -\Delta u = f(u) \quad\hbox{in $\Omega$}
\end{displaymath}
in which the nonlinearity $f$ is such that $f(0)<0$, see \cite{Castro1993}. In particular, $u=0$ is not a subsolution of the equation, and it was remarked in \cite{Lions} that this fact may introduce several difficulties in the analysis of the problem. When the domain $\Omega$ is unbounded the condition $f(0)<0$ is also a source of concern from the viewpoint of variational methods, since non-zero constants are not integrable on $\Omega$.

In literature, second order semipositone problems have been widely studied. In the seminal paper \cite{brown_shivaji}, Brown and Shivaji  considered the perturbed bifurcation problem $-\Delta u = \lambda (u-u^3) -\varepsilon$ on a bounded set $\Omega\subset\mathbb{R}^N$ with Dirichlet boundary conditions and $\varepsilon,\lambda>0$.

 Some years later, Castro and Shivaji  (see~\cite{castroshivaji}) studied non-negative solutions for the boundary value problem
\[
\begin{cases}
	-u''(x) = \lambda f(u(x)), ~x\in (0,1)\\
	u(0) = 0  = u(1),
\end{cases}
\]
where $\lambda>0$ is a constant and $f\in C^2(\mathbb{R})$ satisfies $f(0)<0$.

Subsequently, in \cite{caldwell} the authors considered positive solutions for a class of multiparameter semipositone problem
 \[
 \begin{cases}
 	-\Delta u = \lambda g(u) + \mu f(u) ~&\hbox{in $\Omega$},\\
 	u=0 ~&\hbox{on $\partial\Omega$},
 \end{cases}
 \] 
with $\lambda>0$, $\mu>0$ and $\Omega$ is a smooth bounded domain in $\mathbb{R}^N$, $N\ge 2$. They assume $g(0)>0$
and superlinear while $f(0)<0$ sublinear, and eventually strictly positive. For
fixed $\mu$, they establish existence and multiplicity for $\lambda$ small (when the problem becomes semipositone), and nonexistence
for $\lambda$ large. 

\medskip

To the best of our knowledge Alves, Holanda, and Santos in \cite{alves} were the first to study semipositone problems in the entire space $\mathbb{R}^N$, $N\ge 3$, for the classical Laplace operator, described by the equation
\[
\begin{cases}
	-\Delta u = h(x) \left(f(u) -a\right) ~&\hbox{in $\mathbb{R}^N$, }\\
	u>0 ~&\hbox{in $\mathbb{R}^N$},
\end{cases}
\]
with $f(0) =0$ and $a>0$. The function $f\in C(\mathbb{R}^+)$ is a locally Lipschitz function with subcritical growth which satisfies the Ambrosetti-Rabinowitz condition, and the weight $h$ is a continuous positive function bounded by a radial function $P\in L^1(\mathbb{R}^N)\cap L^\infty(\mathbb{R}^N)$.

 Subsequently, many authors extended their result to other operators. In \cite{biswas2022fourth}, Biswas, Das and Sarkar examined the same problem for a fourth-order operator. Additionally, in \cite{biswas2023study}, Biswas investigated this problem for the fractional Laplacian. In the very recent preprint \cite{biswas2024semipositone}, a semipositone problem with the fractional $p$-Laplace operator was investigated.

\bigskip

In this paper we will be concerned with the following semipositone problem associated with the Grushin operator
\begin{equation}\tag{SP}
	\label{eq:SP}
	-\Delta_\gamma u = g(z) f_a(u)~ \hbox{in $\mathbb{R}^N$}, 
\end{equation}
where the function $g\in L^1(\mathbb{R}^N)\cap L^\infty (\mathbb{R}^N)$ is positive, $a>0$ is a parameter and the function $f_a\colon \mathbb{R}\to\mathbb{R}$ is defined as:
\begin{equation*}
	f_a(t)=
	\begin{cases}
		f(t) -a &\hbox{if $t\ge 0$},\\
		-a(t+1) &\hbox{if $t\in (-1,0)$},\\
		0 &\hbox{if $t\le -1$},
	\end{cases}
\end{equation*}
where $f$ is a contiuous function defined on $[0,+\infty)$ is such that  $f(0)=0$ and
\begin{enumerate}[label=(\textbf{f$_\arabic*$})]
    \item $\lim_{t\to 0^+} \frac{f(t)}{t}=0$, $\lim_{t\to\infty} \frac{f(t)}{t^{\alpha-1}} < +\infty$ for some $\alpha\in (2,2^*_\gamma)$,
	\item (Ambrosetti-Rabinowitz) there exist $\theta>2$ and $t_0>0$ such that
	\begin{displaymath}
		0<\theta F(t) \le tf(t)
	\end{displaymath}
for all $t>t_0$, where $F(t) = \int_{0}^{t} f(s)\,\mathrm{d}s$.
\end{enumerate}
For the reader's convenience, we briefly recall the basic construction of the Baouendi-Grushin operator. We split the euclidean space $\mathbb{R}^N$  as $\mathbb{R}^m \times \mathbb{R}^\ell$, where $m \geq 1$, $\ell \geq 1$ and $m+\ell=N \geq 3$. A generic point $z \in \mathbb{R}^N$ can therefore be written as
\begin{displaymath}
	z = \left(x,y\right) = \left( x_1,\ldots,x_m,y_1,\ldots,y_\ell \right),
\end{displaymath}
where $x \in \mathbb{R}^m$ and $y \in \mathbb{R}^\ell$.
For a differentiable function $u$ we define the Grushin gradient
\begin{displaymath}
	\nabla_\gamma u (z) =  (\nabla_x u(z), \abs{x}^\gamma\nabla_y u(z))= (\partial_{x_1} u(z) , \dots, \partial_{x_m}u(z), \abs{x}^\gamma \partial_{y_{1}}u(z), \dots, \abs{x}^\gamma \partial_{y_{\ell}}u (z)).
\end{displaymath}
The Grushin operator $\Delta_\gamma$ is defined  by\footnote{Some authors use the slightly different definition $\Delta_\gamma = \Delta_x + \left(1+\gamma \right)^2 \vert x \vert^{2\gamma}\Delta_y$.}
\begin{displaymath}
	\Delta_	\gamma u(z) = \Delta_x u(z) + \abs{x}^{2\gamma} \Delta_y u (z),
\end{displaymath} 
where $\Delta_x$ and $\Delta_y$ are the Laplace operators in the variables $x$ and $y$, respectively. A crucial property of the Grushin operator is that it is not uniformly elliptic in the space $\mathbb{R}^N$, since it is degenerate on the subspace $\{0\} \times \mathbb{R}^\ell$.

For a generic $\gamma \geq 0$, it is possible to embed $\Delta_\gamma$ into a family of operators of the form
\begin{displaymath}
	\Delta_\lambda = \sum_{i=1}^N \partial_{x_i} \left( \lambda_i^2 \partial_{x_i} \right),
\end{displaymath}
which have been studied from a geometrical viewpoint in \cite{franchilanconelli1982,franchilanconelli1983,franchilanconelli1984}. Here $\lambda_1,\ldots,\lambda_N$ are given functions which satisfy suitable conditions.

Moreover, the Grushin operator falls into the class of $X$-elliptic operators studied in \cite{gutierrezlanconelli, lanconelli2000x}. In fact, the Grushin operator is uniformly $X$-elliptic with respect to the family of vector fields $X=(X_1,\dots, X_N)$ defined as
\begin{align*}
    X_i &= \frac{\partial}{\partial x_i} \quad\hbox{for $i=1,\ldots,m$} \\
    X_{m+j} &= \vert x \vert^\gamma \frac{\partial}{\partial y_j} \quad\hbox{for $j=1,\ldots,\ell$}.
\end{align*}
With the same notation we can write the Grushin operator $\Delta_\gamma$ as sum of the vector fields
\[
\Delta_\gamma = \sum_{j=0}^{N} X_j^2.
\]
Note that, when $\gamma$ is an even integer, the $C^\infty$-field $X=\left( X_1,\ldots,X_N \right)$ satisfies the H\"{o}rmander's finite rank condition, see \cite{hormander1967}.

\bigskip

We will look for solutions to \eqref{eq:SP} in the function space
\begin{equation}
	D_\gamma(\mathbb{R}^N) = \left\{u\in L^{2^*_\gamma}(\mathbb{R}^N): \norm{u}_\gamma <\infty \right\}, \label{eq:space}
	\end{equation}
	where
\begin{displaymath}
	\norm{u}_\gamma = \left(\int_{\mathbb{R}^N}^{} \abs{\nabla_\gamma u} ^2 \, \dz \right)^{1/2}.	
\end{displaymath}

By definition, a function $u\in D_\gamma (\mathbb{R}^N)$ is  a weak solution to \eqref{eq:SP} if it satisfies the following variational identity:
\begin{displaymath}
	\forall \varphi\in D_\gamma(\mathbb{R}^N) : \quad \int_{\mathbb{R}^N} \nabla_\gamma u \nabla_\gamma \varphi\,\dz = \int_{\mathbb{R}^N} g(z) f_a(u) \varphi\,\dz.
\end{displaymath}
Two functionals $K_a$ and $I_a$ are defined on $D_\gamma(\mathbb{R}^N)$ by
\begin{align*}
	K_a(u) &= \int_{\mathbb{R}^N} g(z) F_a(u(z))\,\dz \\ 
	I_a(u) &= \frac{1}{2} \norm{u}_\gamma^2 - K_a(u).
\end{align*}
It is standard to prove that $K_a$ and $I_a$ are of class $C^1$, and the corresponding Fr\'{e}chet derivatives are given by
\begin{align*}
	&D K_a (u) \colon v\to \int_{\mathbb{R}^N} g(z) f_a(u) v \,\dz ~  \notag\\
	&D I_a (u) \colon v\to \int_{\mathbb{R}^N} \nabla_\gamma u \nabla_\gamma v \,\dz.
\end{align*}
As a consequence, weak solutions of \eqref{eq:SP} are precisely the critical points of $I_a$ on $D_\gamma(\mathbb{R}^N)$.

\bigskip

The magnitude of the constant $a$ and the interplay with the function $g$ are crucial. Specifically, we can prove the existence of positive solutions to \eqref{eq:SP} when the parameter $a$ is sufficiently small.

We will construct  a solution by a mountain pass argument. The positivity of the solution is obtained using the fundamental solution for the Grushin operator.

\bigskip

Our main result can be stated as follows. The distance $d(z)$ appearing in the statement is defined in \eqref{eq:d} below.

\begin{thm}
	\label{thm_main}
Let $f$ satisfy \textbf{(f1)}-\textbf{(f2)} and let $g\in L^1(\mathbb{R}^N)\cap L^\infty(\mathbb{R}^N)$ be positive. Then there exists $a_1>0$ such that for each $a\in (0,a_1)$ \eqref{eq:SP} admits a  weak solution $u_a\in C(\mathbb{R}^N)\cap D_\gamma(\mathbb{R}^N)$ such that $u_a\ge 0$ a.e. in $\{0\}\times\mathbb{R}^\ell$. 
If, in addition,
	\begin{enumerate}[start=3,label={(\bfseries f$_\arabic*$)}]
		\item  $f$ is locally Lipschitz,
	\end{enumerate}
	\begin{enumerate}[start=1,label={(\bfseries g$_\arabic*$)}]
    \item for all $z\in (\{0\} \times \mathbb{R}^\ell)\setminus \{(0,0)\}$, \begin{displaymath}
        d(z)^{N_\gamma -2} \int_{\mathbb{R}^N} \frac{g(\xi)}{d(z-\xi)^{N_\gamma -2}} \mathrm{d}\xi < \infty,
        \end{displaymath} 
	\end{enumerate}
	then there exists $\tilde{a}\in (0,a_1)$ such that $u_a>0$ a.e. in $\{0\} \times \mathbb{R}^\ell$ for every $a\in (0,\tilde{a})$.
\end{thm}

\begin{rem}
	Actually, standard regularity results (see for instance \cite{palatucci2024asymptotic} and the references there-in for the case of the Heisenberg group, but the proof extends to our setting) ensure that $u \in C^{1,\beta}(\mathbb{R}^N)$ with $0<\beta<1$.
\end{rem}

To the best of our knowledge, there are no results in the literature about semipositone problems for degenerate equations of Baouendi-Grushin type. 
Before concluding this introduction, we would like to comment about the positivity in the subspace $\{0\} \times \mathbb{R}^\ell$. In some sense this is an expected result,
since the restriction of $\Delta_\gamma$ to this subspace is a uniformly elliptic operator for which all the standard tools are available. Globally, the Grushin operator
lacks a strong maximum principle for weak solutions and $W^{2,p}$-estimates which are used in order to derive additional regularity for weak solutions. We refer to Remark \ref{rem:concluding}
for more details.

\section{Some inequalities for $f$}

\begin{rem}
    In the rest of the paper we will denote by $C(\ldots)$ a positive constant which depends on the variables between brackets.
\end{rem}

Let us now identify the upper and lower bounds of $f$, $f_a$ and their primitives. Clearly, the function $f_a$ is continuous and the function $F_a$  defined by
\begin{displaymath}
	F_a(t) = 
	\begin{cases}
		F(t) -at ~&\hbox{for $t\ge 0$}, \\
		-\frac{at^2}{2} - at ~&\hbox{for $t\in (-1,0)$}, \\
		\frac{a}{2} ~&\hbox{for $t\le -1$}
	\end{cases}
\end{displaymath}
satisfies $F_a'=f_a$ and $F_a(0)=0$.
The assumption
\begin{align*}
		\lim_{t\to 0^+} \frac{f(t)}{t}=0
\end{align*}
implies that for every $\varepsilon>0$, there exists $t_1(\varepsilon)>0$ such that $f(t)< \varepsilon t$ for $t\in (0,t_1(\varepsilon))$. 
Similarly, assumption
\begin{displaymath}
	\lim_{t\to\infty} \frac{f(t)}{t^{\alpha-1}}<+\infty
\end{displaymath}
implies
\begin{displaymath}
	f(t)\le C(f,t_1(\varepsilon)) t^{\alpha-1} \quad\hbox{for $t\ge t_1(\varepsilon)$}
\end{displaymath}
for some positive constant $C(f,t_1(\varepsilon))$.
	Hence, we have the following bounds for every~$\alpha\in (2,2^*_\gamma]$:
	\begin{equation}
		\label{eq:bounds}
		\abs{f_a(t)} \le \varepsilon\abs{t}+ C(f,t_1(\varepsilon)) \abs{t}^{\alpha -1}+a, \quad F_a(t) \le \varepsilon \abs{t}^2 + C(f,t_1(\varepsilon)) \abs{t}^\alpha + a\abs{t}\quad\hbox{for $t\in\mathbb{R}$}.
	\end{equation}
	By (\textbf{f}$_2$), the function $t \mapsto\frac{F(t)}{t^\theta}$ is non decreasing for $t>t_0$, so there exist two positive constants $M_1$ and $M_2$ such that
	\begin{equation}
		\label{eq:bounds_2}
		F(t) \ge M_1 t^\theta - M_2 \quad\hbox{for all $t\in \mathbb{R}$}.
	\end{equation}
 	\begin{rem}[Ambrosetti-Rabinowitz condition on $f_a$] For $t>t_0$, assumption (\textbf{f}$_2$) yields
		\begin{displaymath}
			\theta F_a(t) = \theta F(t) -\theta a t\le f(t) t - at = t f_a(t).
		\end{displaymath} 
		
		For $t\in [0,t_0]$, by continuity of $F$, there exists a constant $M>0$ independent of $a$, such that
		\begin{displaymath}
			\theta F_a(t) = \theta F(t) -\theta a t \le M -at \le M -a t + f(t)t \le M + tf_a(t)
		\end{displaymath}
		
		For $t\in [-1,0]$,
		\begin{displaymath}
			\theta F_a(t) - tf_a(t) = \left(-\frac{a\theta}{2} +a\right) t^2 - \left(\theta-1\right) at \le c a
		\end{displaymath}
		where $c = \max_{t\in [-1,0]} \left\{\left(-\frac{\theta}{2} +1\right) t^2 - \left(\theta-1\right) t \right\}$ does not depend on $a$.
		
		For $t < -1$,
		\begin{displaymath}
		 \theta F_a(t) = \frac{\theta a}{2} \le \frac{\theta a}{2} + f_a(t) t
		\end{displaymath}
		Choosing $C = \max \left\{c, \frac{\theta}{2} \right\}$ we conclude that the following Ambrosetti-Rabinowitz condition with correction terms holds for $f_a$:
		\begin{equation}
			\label{eq:ARbounds}
			\theta F_a(t) \le t f_a(t) + Ca + M, \quad\hbox{for all $t \in \mathbb{R}$ and $a>0$}.
		\end{equation}
	\end{rem}

\section{A functional setting for the Grushin operator}

The space $D_\gamma(\mathbb{R}^N)$ defined in \eqref{eq:space}
can be also characterized as the completion of $C_0^\infty(\mathbb{R}^N)$ with respect to the norm
\begin{displaymath}
\norm{u}_\gamma = \left(\int_{\mathbb{R}^N}^{} \abs{\nabla_\gamma u} ^2 \, \dz \right)^{1/2},
\end{displaymath}
where
\begin{gather*}
	2_\gamma^* = \frac{2N_\gamma}{N_\gamma -2}
\end{gather*}
is the Sobolev critical exponent associated to $N_\gamma = m + (1+\gamma) \ell$,
the homogeneous dimension associated to the decomposition $N=m+\ell$.
Since the Sobolev-Gagliardo-Niremberg inequality proved in \cite[Lemma 5.1]{loiudice2006}
\begin{displaymath}
	\norm{u}_{2^*_\gamma} \le C \norm{u}_\gamma,
\end{displaymath}
holds for all $u\in D_\gamma(\mathbb{R}^N)$ for some positive constant $C=C(m,k,\gamma)$, we have that the embedding
\begin{equation}
	\label{eq:2}
	D_\gamma(\mathbb{R}^N)\hookrightarrow L^{2^*_\gamma}(\mathbb{R}^N)
\end{equation}
is continuous. The homogeneous space $D_\gamma(\mathbb{R}^N)$ is also a Hilbert space with the inner product\footnote{For the sake of simplicity, we omit any scalar multiplication sign between the two vectors $\nabla_\gamma u$ and $\nabla_\gamma v$.}
\begin{displaymath}
	\braket{u,v}_\gamma = \int_{\mathbb{R}^N}^{} \nabla_\gamma u \nabla_\gamma v\, \, \dz.
\end{displaymath}

Given $q\in [1,\infty)$, with $L^q(\mathbb{R}^N, \abs{g})$ we denote the class of real-valued Lebesgue measurable functions $u$ such that
\[
\int_{\mathbb{R}^N} \abs{g(z)} \abs{u(z)}^q \,\dz<+\infty.
\]
By the assumptions on $g$ and \eqref{eq:2}, we have that the embedding $D_\gamma(\mathbb{R}^N) \hookrightarrow L^q(\mathbb{R}^N, \abs{g})$ is continuous for $q\in [1,2^*_\gamma]$.

The following lemma states a compact embedding of the solution space for \eqref{eq:SP} into the weighted Lebesgue space $L^q(\mathbb{R},\abs{g})$. 
\begin{lemma}
	\label{lemma:1}
	Let $q\in [1,2^*_\gamma)$. The embedding $D_\gamma(\mathbb{R}^N) \hookrightarrow L^q(\mathbb{R}^N, \abs{g})$ is compact.
\end{lemma}
\begin{proof}
	Let $\{u_j\}$ be a bounded sequence in $D_\gamma(\mathbb{R}^N)$. We can suppose that $u_j\rightharpoonup 0$. For each $R>0$, we have the continuous embedding $D_\gamma(\mathbb{R}^N)\hookrightarrow \mathscr{H}_\gamma(B_R(0))$\footnote{For an open bounded set $\Omega\subset \mathbb{R}^N$, we denote by
		$\mathscr{H}_\gamma(\Omega)$ the completion of $C^1_0(\Omega)$ with respect to the norm \begin{displaymath}
			\norm{u}_\gamma = \left(\int_{\Omega}^{} \abs{\nabla_\gamma u} ^2 \, \dz \right)^{1/2}.
			\end{displaymath}}. Since the embedding $\mathscr{H}_\gamma(B_R(0))\hookrightarrow L^q(B_R(0))$ is compact, see \cite[Proposition 3.2]{kogojlanconelli}, it follows that $D_\gamma(\mathbb{R}^N) \hookrightarrow L^q(B_R(0))$ is a compact embedding as well. Hence,
\begin{displaymath}
	u_j(x) \to 0 ~\hbox{a.e. in $\mathbb{R}^N$}
\end{displaymath}
for some subsequence. As $D_\gamma(\mathbb{R}^N) \hookrightarrow L^{2^*_\gamma}(\mathbb{R}^N)$ is a continuous embedding, we have that $\{\abs{u_j}^q\}$ is a bounded sequence in $L^{2^*_\gamma/q}(\mathbb{R}^N)$. By the Brezis-Lieb Lemma, \cite[Theorem 1]{brezislieb}, up to a subsequence if necessary
\begin{displaymath}
	\abs{u_j}^q\rightharpoonup 0~\hbox{in $L^{2^*_\gamma/q}(\mathbb{R}^N)$},
\end{displaymath}
or equivalently,
\begin{displaymath}
	\int_{\mathbb{R}^N}^{} \abs{u_j}^q \varphi \,\dz \to 0, ~\forall\varphi\in L^p(\mathbb{R}^N),
\end{displaymath}
where $p$ is the H\"{o}lder conjugate of $2^*_\gamma/q$. From the assumptions on $g$ we know that $g\in L^r(\mathbb{R}^N)$ for all $r\ge 1$. It follows that
\begin{displaymath}
		\int_{\mathbb{R}^N}^{} g(z)\abs{u_j}^q \,\dz \to 0.
\end{displaymath}
This shows that $u_j\to 0$ in $L^q(\mathbb{R}^N, \abs{g})$, concluding the proof.
\end{proof}

\section{Existence of solutions for (\ref{eq:SP})}
In the next lemma, depending on the values of $a$, we will establish the mountain pass geometry for the functional $I_a$.

\begin{lemma}
	\label{lemma:2}
	Let $N\ge 3$. Let $f$ satisfies \emph{(\textbf{f1})}, \emph{(\textbf{f2})} and $g\in L^1(\mathbb{R}^N)\cap L^\infty(\mathbb{R}^N)$ be positive. Then the following hold:
	\begin{enumerate}[label=(\roman*)]
		\item There exists $r>0$ such that if $\rho\in (0,r)$ and $\norm{u}_\gamma=\rho$, then there are $\delta=\delta(\rho) >0$ and $a_1= a_1(\rho)>0$ such that if $I_a(u)\ge \delta$ for all $a\in (0,a_1)$. 
		\item There exists $v\in D_\gamma(\mathbb{R}^N)$ such that $\norm{v}_\gamma>\rho$ and $I_a(v)<0$, for all $a\in (0,a_1)$.
	\end{enumerate} 

\end{lemma}
\begin{proof}
\begin{enumerate}[label=(\roman*)]
	\item Using \eqref{eq:bounds}, we estimate
	\begin{align*}
		\int_{\mathbb{R}^N} g(z) F_a(u)\, \dz &\le \int_{\mathbb{R}^N} g(z) \left( \varepsilon \abs{u}^2 + C(f,\varepsilon) \abs{u}^\alpha + a\abs{u}\right)\, \dz \\
		&=\varepsilon \int_{\mathbb{R}^N} g(z) \abs{u} ^2\, \dz + C(f,\varepsilon) \int_{\mathbb{R}^N} g(z) \abs{u}^\alpha \,\dz + a \int_{\mathbb{R}^N} g(z) \abs{u} \, \dz \\
		&\le \varepsilon C_1 \norm{u}_\gamma^2 + C_2 \norm{u}_\gamma^\alpha +aC_3 \norm{u}_\gamma,
	\end{align*}
	where $C_1$, $C_2$ and $C_3$ are the embedding constants of $D_\gamma(\mathbb{R}^N) \hookrightarrow L^q(\mathbb{R}^N,\abs{g})$, see Lemma \ref{lemma:1}.
In particular, for $\norm{u}_\gamma = \rho$,
\begin{align}
	I_a(\rho) &\ge \frac{1}{2}\rho^2 -\varepsilon C_1 \rho^2 - C_2 \rho^\alpha - aC_3 \rho\notag \\
	&= \rho^2 \left(\frac{1}{2} - \varepsilon C_1 - C_2 \rho^{\alpha-2}\right) - aC_3 \rho \label{eq:4}.
	\end{align}
We choose $\varepsilon< (2C_1)^{-1}$. Then we write $I_a(u) \ge A(\rho) - aC_3 \rho$, where $A(\rho) = C\rho^2 \left(1-\tilde{C} \rho^{\alpha-2}\right)$, with $C$, $\tilde{C}$ independent of $a$. Let $r$ be the first non-trivial zero of $A$. For $\rho < r$, we fix $a_1\in \left(0,\frac{A(\rho)}{C_3\rho}\right)$ and $\delta= A(\rho) - a_1 C_3\rho$. Thus, by \eqref{eq:4} we get $I_a(u) \ge \delta$ for every $a\in(0,a_1)$.
 \item Fix a function
 \begin{displaymath}
 	\varphi\in C^\infty_0 (\mathbb{R}^N)\setminus \{0\}, \quad \hbox{with $\varphi\ge 0$~ and $\norm{\varphi}_\gamma =1$.} 
 \end{displaymath}
 Notice that for all $t>0$,
 \begin{align*}
 	I_a(t\varphi) = \frac{1}{2} \int_{\Omega} \abs{\nabla_\gamma t\varphi}^2\,\dz - \int_{\Omega} g(z) F_a(t\varphi) \,\dz \\
 	= \frac{t^2}{2} - \int_{\Omega} g(z) F(t\varphi) \,\dz  + a\int_{\Omega} g(z)t\varphi\, \dz
 \end{align*}
 where $\Omega=\operatorname{supp}\varphi$. By \eqref{eq:bounds_2},
\begin{displaymath}
	I_a(t\varphi)\le \frac{t^2}{2} - M_1t^\theta \int_{\Omega} g(z) \abs{\varphi}^\theta \, \dz + t a \norm{g}_\infty + M_2 \norm{g}_1.
\end{displaymath}
Since $\theta>2$ and $a\in (0,a_1)$, we can fix $t_0>1$ large enough so that $I_a(v) <0$, where $v=t_0\varphi\in D_\gamma(\mathbb{R}^N)$.
\end{enumerate}
\end{proof}
As a further step we prove the Palais-Smale condition of $I_a$.
Recall that from the estimate found in \eqref{eq:ARbounds}, for a fixed $a>0$, 
\begin{equation}
	\label{eq:bounds_3}
	\theta F_a(t) \le f_a(t) t + M \quad \hbox{for all $t\in \mathbb{R}$},
\end{equation}
for a suitable positive $M$.
\begin{lemma}
	\label{lemma:3}
	The functional $I_a$ satisfies the Palais-Smale condition for all $a>0$.
\end{lemma}
\begin{proof}
	Let $\{u_j\}_{j}$ be a Palais-Smale sequence in $D_\gamma(\mathbb{R}^N)$. First of all we claim that
	\begin{equation}
		\label{claim:1}
		\text{the sequence $\{u_j\}$ is bounded in $D_\gamma(\mathbb{R}^N)$.}
	\end{equation}
Since $\{u_j\}_j$ is a Palais-Smale sequence, there exists $j_0\in\mathbb{N}$ such that 
		\begin{equation*}
		\left\vert
		\left\langle I'_a (u_j), u_j \right\rangle
		\right\vert \leq \norm{u_j}_\gamma
	\end{equation*}
	for $j\ge j_0$, and there exists $k>0$ such that for any $j\in\mathbb{N}$
	\begin{equation*}
		\abs{I_a(u_j)} \le k.
	\end{equation*}
	So 
	\begin{align}
		\label{eq:6}
		-\norm{u_j}_\gamma- \norm{u_j}_\gamma^2 \le -\int_{\mathbb{R}^N} g(z) f_a(u_j) u_j\,\dz
		\end{align}
	and
		\begin{align}	
		\label{eq:7}
		\frac{1}{2} \norm{u_j}_\gamma^2 - \int_{\mathbb{R}^N} g(z) F_a(u_j) \, \dz \le k
	\end{align}
	for all $j \geq j_0$.
	From \eqref{eq:bounds_3} and \eqref{eq:7},
	\begin{equation}
		\label{eq:8}
		\frac{1}{2} \norm{u_j}_\gamma^2 -\frac{1}{\theta}\int_{\mathbb{R}^N} g(z) f_a(u_j) u_j\, \dz - \frac{1}{\theta} M \norm{g}_1 \le k
	\end{equation}
	for all $j \geq j_0$.
	Thereby, by \eqref{eq:6} and \eqref{eq:8},
	\begin{displaymath}
		\frac{1}{2} \norm{u_j}_\gamma^2 - \frac{1}{\theta} \norm{u_j}_\gamma-\frac{1}{\theta} \norm{u_j}_\gamma^2 \le k +\frac{1}{\theta} M \norm{g}_1,
	\end{displaymath}
	or equivalently,
	\begin{displaymath}
		\left(\frac{1}{2} - \frac{1}{\theta}\right) \norm{u_j}_\gamma^2  -\frac{1}{\theta} \norm{u_j}_\gamma\le k +\frac{1}{\theta} M \norm{g}_1,
	\end{displaymath}
	for $j$ large enough, so \eqref{claim:1} is proved. 

Up to a subsequence, still denoted by $\{u_j\}_j$, there exists $u\in D_\gamma(\mathbb{R}^N)$ such that	
	\begin{displaymath}
	u_j\rightharpoonup u~ \hbox{ in $D_\gamma(\mathbb{R}^N)$}
	\end{displaymath}
and
\begin{displaymath}
	u_j \to u \quad\hbox{a.e. in~ $ \mathbb{R}^N$}.
\end{displaymath}
	From \eqref{eq:bounds} we can estimate
	\begin{displaymath}
		\abs{g(z)f_a(u_j) \left(u_j-u\right)} \le C g(z) \abs{u_j-u} \left(\abs{u_j} + \abs{u_j}^{\alpha-1} + a\right)
	\end{displaymath}
	for some constant $C>0$ independent of $a$.
	Now we we claim that
	\begin{equation}
		\label{claim:2}
		\int_{\mathbb{R}^N}\left( g(z) \abs{u_j - u}\left(\abs{u_j} + \abs{u_j}^{\alpha-1} + a\right)\, \right) \dz \to 0 ~\hbox{as ~$j\to+\infty$}.
	\end{equation}
	In fact, we will only show the limit
\begin{equation}
	\label{eq:9}
			\int_{\mathbb{R}^N}\left( g(z) \abs{u_j - u} \abs{u_j}^{\alpha-1}\, \right) \dz \to 0 ~\hbox{as ~$j\to+\infty$}.
\end{equation}
because the limits involving the other terms follow with the same idea. The sequence $\{\abs{u_j}^{\alpha-1}\}_j$ is bounded in $L^{\frac{2^*_\gamma}{\alpha-1}}(\mathbb{R}^N)$, as \eqref{claim:1} holds. Moreover, applying Lemma \ref{lemma:1} with $q = \frac{2^*_\gamma}{2^*_\gamma - (\alpha -1)}$, we get 
\begin{equation}
	\label{eq:10}
	u_j \to u \in L^{\frac{2^*_\gamma}{2^*_\gamma - (\alpha -1)}} (\mathbb{R}^N,\abs{g}).
\end{equation}
Then, by H\"{o}lder's inequality with conjugate pair $\left(\frac{2^*_\gamma}{\alpha-1}, \frac{2^*_\gamma}{2^*_\gamma - \left(\alpha-1\right)}\right)$,
\begin{displaymath}
		\int_{\mathbb{R}^N} g(z) \abs{u_j - u} \abs{u_j}^{\alpha-1}\,  \dz \le \norm{g\left(u_j-u\right)}_{{2^*_\gamma}/{2^*_\gamma - \left(\alpha-1\right)}} \norm{u_j}_{2^*_\gamma}^{\alpha-1}.
\end{displaymath}
Now, by \eqref{eq:10} we get \eqref{eq:9}.

A consequence of \eqref{claim:2} is the limit
\begin{displaymath}
	\int_{\mathbb{R}^N} g(z) f_a(u_j) (u_j-u) \, \dz \to 0
\end{displaymath}
and so
\begin{equation}
	\label{eq:11}
	\int_{\mathbb{R}^N} \nabla_\gamma u_j \nabla_\gamma \left(u_j - u \right) \, \dz \to 0.
\end{equation}
The weak convergence $u_j\to u$ in $D_\gamma(\mathbb{R}^N)$ yields
\begin{equation}
	\label{eq:12}
		\int_{\mathbb{R}^N} \nabla_\gamma u \nabla_\gamma \left(u_j - u \right) \, \dz \to 0.
\end{equation}
From \eqref{eq:11} and \eqref{eq:12},
\begin{displaymath}
	\int_{\mathbb{R}^N} \abs{\nabla_\gamma u_j - \nabla_\gamma u}^2 \, \dz \to 0,
\end{displaymath}
concluding the proof. 
\end{proof}
Now we are in a position to apply the Mountain Pass Theorem \cite[Theorem 2.1]{ambrosettirabinowitz}, which  ensures the existence and uniform boundedness of solutions to the \eqref{eq:SP} problem.

\begin{thm}
	\label{thm:1}
	Let $a_1>0$ be given in Lemma \ref{lemma:2}. Then for each $a\in (0,a_1)$, \eqref{eq:SP} has a solution $u_a\in D_\gamma(\mathbb{R}^N)$. Moreover, there exists $C>0$ such that $\norm{u}_\gamma \le C$ for all $a\in (0,a_1)$.
\end{thm}
\begin{proof}
 Consider $a_1$, $\delta$ and $v$ as given in Lemma \ref{lemma:2}. For $a\in (0,a_1)$, using Lemma \ref{lemma:2} and \ref{lemma:3}, the hypotheses of the mountain pass theorem \cite[Theorem 2.1]{ambrosettirabinowitz} are verified. So there exists a non-trivial critical point $u_a\in D_\gamma(\mathbb{R}^N)$ of $I_a$ satisfying
 \begin{equation}
 	\label{eq:13}
 	I_a(u_a) = c_a =\inf_{\gamma\in\Gamma_v} 
 	 \max_{t\in [0,1]} I_a(\gamma(t))\ge\delta ~\hbox{and $I'_a(u_a) =0$},
 \end{equation} 
 where $\Gamma_v = \{\gamma\in C\left([0,1], D_\gamma(\mathbb{R}^N)\right)\mid \gamma(0) = 0, \gamma(1) = v\}$. Thus, $u_a$ is a non-trivial solution of \eqref{eq:SP}. To prove the uniform boundedness of $u_a$ in $D_\gamma(\mathbb{R}^N)$, we first show that the set $\{I_a(u_a)\mid a\in (0,a_1)\}$ is uniformly bounded. Considering $t_1$ and $\varphi$ as in Lemma \ref{lemma:2}-(ii), Let us introduce a path $\tilde{\gamma} \colon [0,1] \to D_\gamma(\mathbb{R}^N)$ defined as $\tilde{\gamma}(\sigma) = \sigma v$,
 where $v = t\varphi$ for some $t>t_1$.  From \eqref{eq:13} we have that
 \begin{equation}
 	\label{eq:14}
 	I_a(u_a) \le \max_{\sigma\in [0,1]} I_a(\tilde{\gamma} (\sigma)) = \max_{\sigma\in [0,1]} I_a (\sigma t\varphi).
 \end{equation}
 Now 
 \begin{align}
 	I_a(\sigma t\varphi) &\le \frac{\sigma^2 t^2}{2} \norm{\varphi}_\gamma^2 - M_1 \sigma^\theta t^\theta \int_{\mathbb{R}^N} g(z) \left(\varphi(z)\right)^\theta\, \dz + M_2 \int_{\mathbb{R}^N} g(z) \, \dz + a\sigma t \int_{\mathbb{R}^N} g(z) \varphi(z) \, \dz \notag \\
 	\label{eq:15}
 	&\le \frac{t^2}{2} + M_2 \norm{g}_1 + a_1 t C_1\norm{\varphi}_\gamma.
 \end{align}
 From \eqref{eq:14} and \eqref{eq:15}, there exists $C=C(N_\gamma, M_2,g,a_1)$ such that
 \begin{equation}
 	\label{eq:16}
 	I_a(u_a) \le C, ~\hbox{for all $a\in (0,a_1)$.}
 \end{equation}
 Now we are going to prove that the solutions $u_a$ are uniformly bounded in $D_\gamma(\mathbb{R}^N)$ with respect to $a \in (0,a_1)$. From \eqref{eq:14} we know that
 \begin{equation}
 	\label{eq:17}
 	\norm{u_a}_\gamma^2 - \int_{\mathbb{R}^N} g(z) f_a(z) u_a \, \dz = 0.
 \end{equation}
Also, by \eqref{eq:16} we have
\begin{equation}
	\label{eq:18}
	\frac{1}{2}\norm{u_a}_\gamma ^2 - \int_{\mathbb{R}^N} g(z) F_a (u_a) \, \dz \le C.
\end{equation}
Now first multiplying \eqref{eq:17} by $\frac{1}{\theta}$ and then subtracting into \eqref{eq:18} gives the following
\begin{displaymath}
	\left(\frac{1}{2} - \frac{1}{\theta}\right) \norm{u_a}_\gamma^2 + \int_{\mathbb{R}^N} g(z)\left(\frac{1}{\theta}f_a(u_a) u_a - F_a(u_a)\right) \, \dz \le C,
\end{displaymath}
and combining the above with \eqref{eq:bounds_3} yields
\begin{displaymath}
	\left(\frac{1}{2} - \frac{1}{\theta}\right) \norm{u_a}_\gamma^2 - \frac{1}{\theta} M_3 \norm{g}_1 \le C.
\end{displaymath}
  So, there exists $C$ such that $\norm{u_a}_\gamma \le C$ for every $a\in (0,a_1)$.
  \end{proof} 
  \section{Properties of the solutions}
Our next result establishes that the solution $u_a$ belongs to $L^\infty(\mathbb{R}^N)$ if the parameter $a$ is sufficiently small. Using the iterative argument of \cite[Lemma B.3]{struwe} we prove that the solution $u_a\in L^q(\mathbb{R}^N)$, for every $q\in [2^*_\gamma,\infty)$, provided $a$ is small enough. Moreover the argument in \cite[Theorem 3.1]{loiudice2009} can be applied to establish the boundedness of $u_a$.
\begin{lemma}
\label{lemma:4}
	There exists $a_2\in (0,a_1)$ such that $u_a\in L^\infty(\mathbb{R}^N)\cap C(\mathbb{R}^N)$ for all $a\in (0,a_2)$.
\end{lemma}
  \begin{proof}
  	In order to prove the lemma, it is enough to show that for any sequence $a_j\to 0$, the sequence of solutions $u_j = u_{a_j}$ possesses a subsequence, still denoted by itself, which is bounded in $L^\infty(\mathbb{R}^N)$. By Lemma \ref{thm:1}, the sequence $\{u_j\}$ is bounded in $D_\gamma(\mathbb{R}^N)$, then for some subsequence, there is $u\in D_\gamma(\mathbb{R}^N)$ such that
  	\begin{displaymath}
  		u_j \rightharpoonup u \quad\hbox{in $D_\gamma(\mathbb{R}^N)$}
  	\end{displaymath}
  	and
  	\begin{displaymath}
  		u_j \to u \quad\hbox{a.e. in $\mathbb{R}^N$}.
  	\end{displaymath}
  	With the same approach explored in the proof of Lemma \ref{lemma:3}, we have that 
  	\begin{displaymath}
  		u_j\to u \quad\hbox{in $D_\gamma(\mathbb{R}^N)$}.
  	\end{displaymath} 
  	Consequently,
  	\begin{equation}\
  		\label{eq:19}
  		u_j\to u \quad\hbox{in $L^{2^*_\gamma}(\mathbb{R}^N)$}
  	\end{equation}
  	and for some subsequence, there is $h\in L^{2^*_\gamma}(\mathbb{R}^N)$ such that
  	\begin{displaymath}
  		\abs{u_j(z)}\le h(z), \quad \hbox{a.e. in $\mathbb{R}^N$ and $\forall j\in\mathbb{N}$}.
  	\end{displaymath}
  	Setting the function
  	\[
  	V_j(z) = g(z) \frac{\left(1+\abs{u_j}^{\frac{N_\gamma +2}{N_\gamma -2}}\right)}{1+\abs{u_j}},
  	\]
  	it follows that
  	\begin{equation}
  		\label{eq:20}
  		\abs{V_j}\le g(1+\abs{u_j}^{4/{N_\gamma-2}}) \le g (1+\abs{h}^{4/{N_\gamma -2}}), \quad \forall j\in\mathbb{N}.
  	\end{equation}
  	By the assumption on the function $g$, $V_j\in L^p(\mathbb{R}^N)$, for every $p\in [1,N_\gamma/2]$. Thereby, the Dominated Convergence Theorem with \eqref{eq:19} and \eqref{eq:20} implies that $V_j\to V$ in $L^{p}(\mathbb{R}^N)$, for some $V\in L^{p}(\mathbb{R}^N)$, $p\in [1,2^*_\gamma]$. So there exist a constant $C>0$ and $j_0\in \mathbb{N}$ such that

\begin{displaymath}
	\abs{g(z) f_{a_j}(u_j)} \le C V(z) \left(1+\abs{u_j}\right), \quad \forall j\ge j_0.
\end{displaymath}
In the rest of the proof we will assume $C=1$ to save notation.

\medskip

\textbf{Step 1.} Given $q\in [2^*_\gamma,\infty)$, there exists a constant $K_q$ such that $\norm{u_j}_q \le K_q$ for all $j\ge j_0$.

Let us prove that if $u_j\in L^q(\mathbb{R}^N)$ for some $q\ge 2^*_\gamma$, then $u_j\in L^{\frac{2^*_\gamma}{2} q}$.  Let $L\ge 0$ and $G(u_j) = G_L(u_j) = u_j \min\{{u_j^ {(q/2)-1}, L}\}$ and define 
\[
F(u_j) = F_L(u_j) = \int_{0}^{u_j} \abs{G'(t)}^2 \,\dt \in D_\gamma(\mathbb{R}^N).
\]
First of all, observe that
\begin{equation}
	\label{eq:21}
	\int_{\mathbb{R}^N} \nabla_\gamma F(u_j) \nabla_\gamma u_j \,\dz = \int_{\mathbb{R}^N} \abs{\nabla_\gamma G(u_j)}^2 \,\dz ~\hbox{for every $u_j\in D_\gamma(\mathbb{R}^N)$}.
\end{equation}
and the following growth condition holds
\begin{displaymath}
	u_j F(u_j) \le C_q G(u_j)^2 \le C_q u_j^q,
\end{displaymath}
with a constant $C_q$ that can be explicitly given by $C_q = \frac{q^2}{4(q-1)}$.

 Using $F(u_j)$ as a test function in \eqref{eq:SP} we get
 \begin{multline}
 	\label{eq:22}
 	\int_{\mathbb{R}^N} \nabla_\gamma F(u_j) \nabla_\gamma u_j \, \dz \le \int_{\mathbb{R}^N} V(z) (1+\abs{u_j}) u_j\min \{u_j^{q-2}, L^2 \}\,\dz \\ \le \int_{\mathbb{R}^N} V(z) (1+2\abs{u_j}^2) \min \{u_j^{q-2}, L^2 \}  \,\dz  \\
 	\le \int_{\mathbb{R}^N} V(z)\, \dz + 3 \int_{\mathbb{R}^N} V(z) \abs{u_j}^2 \min \{u_j^{q-2}, L^2 \} \,\dz.
 \end{multline}
 Moreover, if $q\in L^q(\mathbb{R}^N)$, then for any $K\ge 1$  by \eqref{eq:21} and \eqref{eq:22} there holds
 \begin{multline*}
 	\int_{\mathbb{R}^N} \abs{\nabla_\gamma G(u_j)}^2 \,\dz \le c + 3\int_{\mathbb{R}^N} V(z) \abs{u_j}^2 \min \{u_j^{q-2}, L^2 \}\,\dz \\\le c+ 3 K\int_{\mathbb{R}^N} \abs{u_j}^2 \min \{u_j^{q-2}, L^2 \} \,\dz + 3 \int_{\{z\in\mathbb{R}^N\mid V(z) \ge K\}} \abs{u_j}^2 \min \{u_j^{q-2}, L^2 \}\,\dz\\\le
 	c (1+K) + c\left(\int_{\{z\in\mathbb{R}^N\mid V(z) \ge K\}} V(z)^{N_\gamma/2} \,\dz \right)^{N_\gamma/2} \times \left(\int_{\mathbb{R}^N} \abs{G(u_j)} ^{\frac{2N_\gamma}{N_\gamma -2}} \,\dz\right)^{\frac{N_\gamma - 2}{N_\gamma}}\\ \le
 	c(1+K) + c \varepsilon(K) \int_{\mathbb{R}^N} \abs{\nabla_\gamma G(u_j)}^2\,\dz,
 \end{multline*}
 with some constants $c$ depending on the $L^q(\mathbb{R}^N)$ norm of $u_j$ and
 \[
 \varepsilon(K) = \left(\int_{\{z\in\mathbb{R}^N\mid V(z) \ge K\}} V(z)^{N_\gamma/2}\right)^{2/N_\gamma} \to 0 \quad \hbox{as $K\to \infty$}.
 \]
 Fix $K$ such that $\varepsilon(K) \le 1/2$, for this choice of $K$ we may conclude that
 \[
 \int_{\mathbb{R}^N}\abs{\nabla_\gamma G(u_j)}^2\, \dz 
 \]
 remains uniformly bounded in $L$. Hence we may let $L\to \infty$ to derive that $u_j\in L^{\frac{2^*_\gamma}{2} q}(\mathbb{R}^N)$. We can conclude our argument by iterating the process.
 
 \medskip
 
 \textbf{Step 2.} There exists $C>0$ such that $\norm{u_j}_\infty\le C$ for all $j\ge j_0$.
  
Since, by the previous Step, $u_j\in L^q(\mathbb{R}^N)$ for all $q\in [1,\infty)$ then also $V\in L^q(\mathbb{R}^N)$ for every $q\in [1,\infty)$. Let $t_0>N_\gamma/2$, 
we  construct a sequence of $L^{q_k}(\mathbb{R}^N)$ norms of $u_j$, with $q_k\to\infty$, which are uniformly bounded by the $L^{q_0}(\mathbb{R}^N)$ norm of $u_j$, where $q_0 = t_0'2^*_\gamma$ and $t_0'$ is such that $1/t_0+ 1/t_0' = 1$.  
As before, using again the test function $F(u_j)$, we estimate the right-hand side of \eqref{eq:SP} as follows
   \begin{multline}
   	\label{eq:23}
 	\int_{\mathbb{R}^N} V(z) (1+\abs{u_j}) F(u_j)\,\dz \le \norm{V}_{t_0} \norm{u_j F(u_j)}_{t_0'}  + C_q\norm{V}_{\tilde{t}_0}\norm{u_j}_{qt_0'}^{q-1}\\
 	\le C_q \norm{u_j}_{q t_0'}^q \left(\norm{V}_{t_0} + \norm{V}_{\tilde{t}_0} \norm{u_j}_{qt_0'}^{-1}\right) = \tilde{C} C_q \norm{V}_{t_0} \norm{u_j}_{qt_0'}^q
  \end{multline} 
 for some $\tilde{t}_0>N_\gamma/2$ and the constant $\tilde{C} = \tilde{C}(q, V, u_j)$. Hence, by \eqref{eq:21} and \eqref{eq:23}, we obtain
 \begin{displaymath}
 	\int_{\mathbb{R}^N} \abs{\nabla_\gamma G(u)}^2\,\dz \le \tilde{C} C_q \norm{V}_{t_0}\norm{u_j}_{qt_0'}^q.
 \end{displaymath}
 By Sobolev inequality and letting $L\to\infty$, we get
 \begin{equation*}
 	\norm{u_j}_{\frac{2^*_\gamma}{2}q}^q \le C C_q \norm{V}_{t_0} \norm{u_j}_{qt_0'}^q.
 \end{equation*}
 Let $\delta = \frac{2^*_\gamma}{2t_0'}>1$. With this notation, the preceding estimate can be rewritten as
  \[
 \norm{u_j}_{\delta q t_0'} \le \left(C C_q\right)^{1/q} \norm{V}_{t_0}^{1/q}\norm{u_j}_{qt_0'}.
 \]
 Now, define $q_0  = 2^*_\gamma t_0'$ and $q_k = \delta^k q_0$ for $k\ge 1$. Observing that $C_q\le C q$, we obtain
 \[
 \norm{u_j}_{q^k} \le \left( \prod_{i=0}^{k-1} \left(C q_i\right) ^{1/q_i}\right) \norm{V}_{t_0}^{\sum_{i=0}^{k-1}\frac{1}{q_i}} \norm{u_j}_{q_0}
 \]
 whose right-hand side is finite since
 \begin{displaymath}
 	\sum_{i=0}^{\infty} \frac{1}{q_i} = \frac{1}{q_0} \sum_{i=0}^{\infty} \frac{1}{\delta_i} <\infty \quad \text{and} \quad \sum_{i=0}^{\infty} \frac{\log q_i}{q_i}<\infty.
 \end{displaymath}
 Therefore, letting $k\to\infty$, we can conclude.

 \textbf{Step 3}. 
	The continuity of the solutions $u_j$ comes from the non-homogeneous Harnack inequality for $X$-elliptic operators established in \cite[Theorem 5.5]{gutierrezlanconelli} 
	combined with some standard arguments, see for example \cite[Corollary 4.18]{han} for further details.
\end{proof}
  
In the next two lemmas, we prove that the solution $\{u_a\}$ is uniformly bounded from below over several spaces.
\begin{lemma}
	Let $a_1$ as in Theorem \ref{thm:1}. Then there exists $C_1>0$ such that $\norm{u_a}_\gamma \ge C_1$ for all $a\in(0,a_1)$.
\end{lemma}  
\begin{proof}
	We notice that $F_a(t)\ge -a\abs{t}$ for all $t\in\mathbb{R}$. For $\delta$ as given in Theorem \ref{thm:1} and from \eqref{eq:13} we write $I_a(u_a)\ge\delta$, for all $a\in(0,a_1)$. Using the embedding of $D_\gamma(\mathbb{R}^N)$ proved in Lemma \ref{lemma:1} we have
	\begin{equation*}
		\delta \le I_a(u_a) \le \frac{1}{2} \norm{u_a}_\gamma^2 + a\int_{\mathbb{R}^N} g(z) \abs{u_a(z)} \, \dz \le \frac{1}{2}\norm{u_a}_\gamma^2 + a_1 C\norm{u_a}_\gamma.
	\end{equation*}
	So there exists $C_1 = C_1(N,\gamma,g,a_1,\delta)$ such that $\norm{u_a}_\gamma \le C_1$ for all $a\in (0,a_1)$.
\end{proof}
  In what follows, we show an estimate from below to the norm $L^\infty(\mathbb{R}^N)$ of $u_a$ for $a$ small enough.
  \begin{lemma}
  	\label{lemma:5}
  	There exists $a_3\in (0,a_2)$ and $\beta>0$ that does not depend on $a\in (0,a_3)$, such that $\norm{u_a}_\infty\ge\beta$ for all $a\in(0,a_3)$.
  \end{lemma}
  \begin{proof}
    	For the constant~$\delta$ given in Theorem \ref{thm:1}, from \eqref{eq:13} we write $I_a(u_a)\ge \delta$ for all $a\in (0,a_1)$.
  	Further, we have (notice again that $F_a(t) \ge -a\abs{t}$ for all $t\in\mathbb{R}$),
  	\begin{equation*}
  		\frac{\norm{u_a}_\gamma^2}{2} = I_a(u_a) + \int_{\mathbb{R}^N} g(z) F_a(u_a) \, \dz \ge \delta - a\int_{\mathbb{R}^N} g(z) \abs{u_a} \, \dz.
  	\end{equation*}
  	For every $a\in (0,a_1)$, using the continuous embedding $D_\gamma(\mathbb{R}^N)\hookrightarrow L^1(\mathbb{R}^N, \abs{g})$, Lemma \ref{lemma:1}, with embedding constant $C_2 = C_2(N,\gamma,g)$ and the uniform boundeness of $\{u_a\}$ in $D_\gamma(\mathbb{R}^N)$, Theorem \ref{thm:1}, we obtain
  	\begin{equation*}
  		\frac{\norm{u_a}_\gamma^2}{2} \ge \delta - a C_2 \norm{u_a}_\gamma\ge \delta - aC_2 C = \delta - aC_3,
  	\end{equation*}
  	where $C_3 = C_2C$. Now if we choose $a_3$ such that $0<a_3<\min\left\{\frac{\delta}{C_3}, a_2\right\}$, then
  	\begin{equation}
  		\label{eq:25}
  			\frac{\norm{u_a}_\gamma^2}{2} \ge \delta_0 = \delta - a_3 C_3 >0, \quad\forall a\in (0,a_3).
  	\end{equation}
  		Since $u_a$ is a weak solution of \eqref{eq:SP}, then 
  	\begin{equation*}
  		\int_{\mathbb{R}^N} \nabla_\gamma u_a \nabla_\gamma \varphi\,\dz  = \int_{\mathbb{R}^N} g(z) f_a(u_a)\varphi\,\dz, \quad \forall \varphi\in D_\gamma(\mathbb{R}^N).
  	\end{equation*}
For  $\varphi=u_a$ we have
\begin{equation*}
	\int_{\mathbb{R}^N} \abs{\nabla_\gamma u_a}^2 \, \dz =\int_{\mathbb{R}^N} g(z) f_a(u_a) u_a \, \dz
\end{equation*}
So, by the estimation found in \eqref{eq:25}
\begin{equation*}
	\int_{\mathbb{R}^N} g(z) f_a(u_a) u_a \, \dz \ge 2\delta_0>0,\quad \forall a\in(0,a_3).
\end{equation*}
Thus, \eqref{eq:bounds} gives
\begin{equation*}
	2\delta_0 \le \int_{\mathbb{R}^N} g(z) \left( C \left(\abs{u_a}^{\alpha-1} + \abs{u_a}\right) +a\right)\, \dz \le \left( C \left( \norm{u_a}_\infty^{\alpha-1} + \norm{u_a}_\infty \right)+ a\right)\norm{g}_1.
\end{equation*}
This implies that $\norm{u_a}_\infty\ge \beta$ for some $\beta>0$ for all $a\in (0,a_3)$, taking $a_3$ smaller if necessary.
  \end{proof}

\section{Positive solutions}
As in the previous Section, given a sequence $a_j\to 0$ as $j\to\infty$, let us consider the solution $u_j = u_{a_j}$ for \eqref{eq:SP}. We denote by $u$ the limit of $u_j$ as in Lemma \ref{lemma:4}.

Now consider the function $f_0$
\[
f_0(t) = 
\begin{cases}
	f(t), \quad &\hbox{if $t\ge 0$};\\
	0,\quad &\hbox{if $t<0$}.
\end{cases}
\]
\begin{prop}
	\label{prop:1}
 $u$ is a weak solution of
\begin{equation}
	\label{eq:26}
	-\Delta_\gamma u = g(z) f_0(u)\quad \hbox{in $\mathbb{R}^N$}. 
\end{equation}
Moreover, $u\in L^p(\mathbb{R}^N)\cap C(\mathbb{R}^N)$ for every $p\in [2^*_\gamma,\infty]$.
\end{prop}
\begin{proof}
	Since $a_j\to 0$, there exists $j_1\in\mathbb{N}$ such that $a_j\le a_2$ for all $j\ge j_1$ (with $a_2$ we denote the parameter $a_2$ as in the statement of Lemma \ref{lemma:4}). 
	Now, for every $\varphi\in D_\gamma(\mathbb{R}^N)$ with $\varphi\ge 0$ we show that
	\[
	\int_{\mathbb{R}^N} g(z) f_j (u_j) \varphi(z)\,\dz \to \int_{\mathbb{R}^N} g(z) f_0(u)\varphi(z)\,\dz, \quad\hbox{as $j\to \infty$}.
	\]
We split
\[
\abs{f_j(u_j) - f_0(u)} \le \abs{f_j(u_j) - f_0 (u_j)} + \abs{f_0(u_j) - f_0(u)}.
\]
By the continuity of $K_0'$, 
\[
\int_{\mathbb{R}^N} g(z) \abs{f_0(u_j) - f_0(u)} \varphi(z)\, \dz \to 0, \quad\hbox{as $j\to \infty$}.
\]
Further, $\abs{f_j(u_j) - f_0(u_j)} \le a_j$. Therefore,
\begin{multline*}
	\int_{\mathbb{R}^N} g(z) \abs{f_j(a_j) - f_0(u)} \varphi(z)\,\dz \\
	\le a_j \int_{\mathbb{R}^N}g(z)\varphi(z)\,\dz + \int_{\mathbb{R}^N} g(z) \abs{f_0(u_j) - f_0(u)} \varphi(z)\,\dz \to 0\quad\hbox{as}~ j\to\infty.
\end{multline*}
From the weak formulation 
\begin{displaymath}
	\int_{\mathbb{R}^N} \nabla_\gamma u_j (z) \nabla_\gamma \varphi(z)\,\dz = \int_{\mathbb{R}^N} g(z) f_j(u_j) \varphi(z)\,\dz, \forall\varphi\in D_\gamma(\mathbb{R}^N), \quad\varphi\ge 0.
\end{displaymath}
Taking the limit as $j\to\infty$ gives
\begin{displaymath}
	\int_{\mathbb{R}^N} \nabla_\gamma u(z) \nabla_\gamma \varphi(z)\,\dz = \int_{\mathbb{R}^N} g(z) f(u) \varphi(z)\,\dz, \forall\varphi\in D_\gamma(\mathbb{R}^N), \quad\varphi\ge 0.
\end{displaymath}
Now, for any $\varphi\in D_\gamma(\mathbb{R}^N)$, we write $\varphi = \varphi^+ - \varphi^-$. We see that the above identity holds for both $\varphi^+$ and $\varphi^-$, so
\begin{equation}
	\label{eq:26b}
	\int_{\mathbb{R}^N} \nabla_\gamma u(z) \nabla_\gamma \varphi(z)\,\dz = \int_{\mathbb{R}^N} g(z) f(u) \varphi(z)\,\dz,\quad \forall\varphi\in D_\gamma(\mathbb{R}^N),
\end{equation}
 that is, $u$ is a weak solution of \eqref{eq:26}.
 With the same argument of Lemma \ref{lemma:4}, we obtain $u\in L^p(\mathbb{R}^N)$, $p\in [2^*_\gamma,\infty]$ and $u$ is a continuous function in $\mathbb{R}^N$.
\end{proof}

\begin{prop}
\label{prop:positive}
	The function $u$ is positive.
\end{prop}

\begin{proof}
	Taking the test function $\varphi = u^- = \min \left\{u,0 \right\}$ in \eqref{eq:26b} we find that $\norm{u^-}_\gamma = 0$, so $u\ge 0$ a.e. in $\mathbb{R}^N$. 
	
	Let $\Sigma = \{0\} \times \mathbb{R}^\ell$ be the degenerate set for the Grushin operator.
	If $\Omega$ is any bounded domain contained in $\mathbb{R}^N \setminus \Sigma$, then $\Delta_\gamma$ is a uniformly elliptic operator on $\Omega$. Hence $u \in C^2(\Omega)$ and either $u \equiv 0$ on $\Omega$ or $u>0$ in $\Omega$. In the former case, it follows from \cite[Proposition 2.5]{kogojlanconelli} that $u \equiv 0$ on the connected component of $\mathbb{R}^N \setminus \Sigma$ which contains $\Omega$. Therefore $u \equiv 0$ on $\mathbb{R}^N \setminus \Sigma$.
	
	Now, it is clear that $\Sigma$ has empty interior, and it follows from \cite[Theorem 9, Chapter 1]{kelley} that $\mathbb{R}^N \setminus \Sigma$ is dense in $\mathbb{R}^N$. Recalling that $u$ is a continuous function, we conclude that $u \equiv 0$ on $\Sigma$. To this stage the function $u$ might be zero in (the interior of) some connected component of $\mathbb{R}^N \setminus \Sigma$ and strictly positive in (the interior of) some other component. However this is again impossible, since \cite[Proposition 2.5]{kogojlanconelli} applies again on any ball which meets two such components across $\Sigma$. Hence either $u \equiv 0$ on $\mathbb{R}^N$, or $u>0$ on $\mathbb{R}^N$. The former is not admissible due to the Lemma \ref{lemma:5}.
\end{proof} 
	
\begin{rem}
When $\gamma$ is an even integer, a very elegant proof of the positivity of $u$ can be devised. We follow an idea of \cite{biagi2022sublinear}. When $\gamma$ is an integer, it is known that $-\Delta_\gamma$ satisfies the H\"{o}rmander condition. Recalling that $g>0$ and $f_0 \geq 0$, we see that $\int_{\mathbb{R}^N} \nabla_\gamma u \nabla_\gamma \varphi \, \dz \geq 0$ for every $\varphi \in C_0^\infty(\mathbb{R}^N)$ such that $\varphi \geq 0$. Hence $\int_{\mathbb{R}^N} u (-\Delta_\gamma \varphi) \, \dz \geq 0$ for every such $\varphi$. Recalling that $u$ is a continuous function, it follows that $u$ is a viscosity subsolution of the operator $-\Delta_\gamma$. The strong maximum principle proved in \cite[Corollary 1.4]{bardi2019new} yields that either $u \equiv 0$ or $u>0$ in $\mathbb{R}^N$. The first option, again, is not admissible due to the Lemma \ref{lemma:5}.
\end{rem}

\

Now we are going to show that $u_j\to u$ in $L^\infty(\{0\} \times \mathbb{R}^\ell)$. We essentially follow \cite{alves, biswas2023study}, although a few important differences will arise.

For $z = (x,y) \in \mathbb{R}^N = \mathbb{R}^m\times\mathbb{R}^\ell$, we denote by
\begin{equation}
	\label{eq:d}
d(z) = \left( \abs{x}^{2\left(\gamma +1\right)} + (\gamma+1)^2 \abs{y}^2\right)^{\frac{1}{2\left(\gamma+1\right)}}
\end{equation}
the homogeneous norm associated to $\Delta_\gamma$. Note that this norm is homogeneous of degree $1$ with respect to the anisotropic dilatation
\begin{displaymath}
	\delta_\lambda(x,y) = (\lambda x, \lambda^{\gamma+1}y),\quad \lambda >0.
\end{displaymath}
and, for $\gamma=0$, it reduces to the usual Euclidean distance in $\mathbb{R}^N$.

The function 
\begin{equation*}
	\Gamma(z) = \frac{C}{d(z)^{N_\gamma -2}}
\end{equation*}

is a fundamental solution for the Grushin operator with singularity at $z=0$, see  \cite{beckner} or \cite[Appendix B]{ambrosio2003}.

The constant $C$ is a suitable positive constant that depends on $m,\ell$ and $\gamma$. Precisely,
\[
C^{-1} = (N_\gamma -2) \int\limits_{\{z \mid d(z) = 1\}} \frac{\abs{x}^{2\gamma}}{\left(\abs{x}^{2\left(1+2\gamma\right)}+ \left(1+\gamma\right)^2 \abs{y}^2\right)^{1/2}} \,\mathrm{d}S,
\]
where $S$ denotes the surface measure of the manifold $\left\{ z \mid d(z)=1 \right\}$.

We recall the following definition.
\begin{definizione}
	A function $u\colon\mathbb{R}^N\to\mathbb{R}$ is radially symmetric (radial for short) if there exists a function $v\colon\mathbb{R}\to\mathbb{R}$ such that
	\[
	u(z) = v(r)
	\]
	where $r = d(z)$, for all $z\in\mathbb{R}^N$.
\end{definizione}
For $R>0$ we set $B_R = \left\{ \xi\in\mathbb{R}^N \mid d(\xi) < R \right\}$.

Let us consider a spherical shell $\Omega = B_{R_2}\setminus \overline{B_{R_1}}$, with $0\le R_1 < R_2\le +\infty$. If $u$ is a radial function, by \cite[Eq. (2.4)]{ambrosio2003hardy} we get the formula
\begin{equation}
	\label{eq:integral}
	\int_{\Omega} u(d(\xi)) \, \mathrm{d}\xi = s_N \int_{R_1}^{R_2} \rho^{N_\gamma -1} u(\rho) \,\mathrm{d}\rho,
\end{equation}
 where $s_N$ is an explicit constant that depends only on $N$.

 \begin{prop}
 \label{prop:uniform}
 	$\norm{u_j - u}_{L^\infty(\{0\} \times \mathbb{R}^\ell)} \to 0$ as $j\to \infty$. 
 \end{prop}
\begin{proof}
By \cite[Sect. 3.1]{dambrosiomitidieripohozaev} and recalling that $u$ is a continuous function, we have the representation formula
\begin{equation}
		\label{eq:27}
		u(z) = C \int_{\mathbb{R}^N} \frac{g(\xi) f_0 (u(\xi))}{d(z-\xi)^{N_\gamma -2}} \, \mathrm{d}\xi \quad\hbox{for a.e. $z \in \{0\} \times \mathbb{R}^\ell$}.
	\end{equation}
		Moreover, since $u_j$ is a solution of \eqref{eq:SP}, we also have 
			\begin{equation}
				\label{eq:28}
			u_j(z) = C \int_{\mathbb{R}^N} \frac{g(\xi) f_j (u_j(\xi))}{d(z-\xi)^{N_\gamma -2}} \, \mathrm{d}\xi \quad\hbox{for a.e. $z \in \{0\} \times \mathbb{R}^\ell$}.
		\end{equation}
Using \eqref{eq:27} and \eqref{eq:28} we estimate $\abs{u-u_j}$ as follows: for every $z \in \{0\} \times \mathbb{R}^\ell$,
\begin{multline}
	\label{eq:29}
	\abs{u(z)-u_j(z)} \le C \int_{\mathbb{R}^N} g(\xi) \frac{\abs{f_j(u_j(\xi)) - f_0(u(\xi))}}{d(z-\xi)^{N_\gamma-2}} \, \mathrm{d}\xi \\ \le C \left(
	\int_{B_1} g(\xi) \frac{\abs{f_j(u_j(\xi)) - f_0(u(\xi))}}{d(z-\xi)^{N_\gamma-2}} \, \mathrm{d}\xi + \int_{\mathbb{R}^N\setminus B_1} g(\xi) \frac{\abs{f_j(u_j(\xi)) - f_0(u(\xi))}}{d(z-\xi)^{N_\gamma-2}} \, \mathrm{d}\xi \right).
\end{multline}
Take $\delta>1$. Applying the H\"{o}lder's inequality with the conjugate pairs $(\delta,\tilde{\delta})$ we estimate the first integral of \eqref{eq:29} as
\begin{multline}
	\label{eq:30}
	\int_{B_1} g(\xi) \frac{\left| f_j(u_j(\xi)) - f_0(u(\xi)) \right|}{d(z-\xi)^{N_\gamma-2}} \, \mathrm{d}\xi \le \left( \int_{B_1} \frac{1}{d(z-\xi)^{(N_\gamma -2)\delta}} \, \mathrm{d}\xi \right)^{\frac{1}{\delta}} \times \\
	\left( \int_{B_1} g(\xi)^{\tilde{\delta}} \left| f_j(u_j(\xi)) - f_0(u(\xi)) \right|^{\tilde{\delta}} \, \mathrm{d}\xi \right)^{\frac{1}{\tilde{\delta}}}.
\end{multline}

We compute, using \eqref{eq:integral}
\[
\int_{B_1} \frac{1}{d(z-\xi)^{(N_\gamma -2)\delta}} \, \mathrm{d}\xi = s_N \int_{0}^{1} \frac{\rho^{N_\gamma -1}}{\rho^{(N_\gamma-2)\delta}} d\rho \le C(N_\gamma).
\]
Moreover, the second term in the inequality \eqref{eq:30} converges to zero. In fact, proceeding as in Lemma \ref{lemma:4}, $u_j\to u$ in $L^p,\, p\in [2^*_\gamma,\infty)$. Then we can apply the Dominated Convergence Theorem.

Next, the second integral of \eqref{eq:30} has the following bound:
\begin{displaymath}
	\int_{\mathbb{R}^N\setminus B_1} g(\xi) \frac{\abs{f_j(u_j(\xi)) - f_0(u(\xi))}}{d(z-\xi)^{N_\gamma-2}} \, \mathrm{d}\xi \le \int_{\mathbb{R}^N\setminus B_1} g(\xi) \abs{f_j(u_j(\xi)) - f_0(u(\xi))}\,\mathrm{d}\xi.
\end{displaymath}
Again, by the generalized dominated convergence theorem, 
\begin{displaymath}
\lim_{j \to +\infty} \int_{\mathbb{R}^N\setminus B_1} g(\xi) \frac{\abs{f_j(u_j(\xi)) - f_0(u(\xi))}}{d(z-\xi)^{N_\gamma-2}} \, \mathrm{d}\xi = 0.
\end{displaymath}
Therefore, \eqref{eq:29} yields $u_j\to u$ in $L^\infty(\{0\} \times \mathbb{R}^\ell)$ as $j\to \infty$.
\end{proof}
The immediate corollary concludes the proof of the first statement of Theorem \ref{thm_main}.
\begin{cor}
	$u_j\ge 0$ a.e. in $\{0\} \times \mathbb{R}^\ell$ for $j$ large enough.
\end{cor}

Let us now conclude the proof of Theorem \ref{thm_main}.
\begin{thm}
	If \textbf{(f3)} and \textbf{(g1)} are satisfied then $u_j>0$ a.e. in $\{0\} \times \mathbb{R}^\ell$.
\end{thm}
\begin{proof}
	For each $j\in\mathbb{N}$, since $f_j$ is locally Lipschitz \textbf{(f3)} and $0\le u_j,u\le C$, we have
	\begin{displaymath}
		\abs{f_j(u_j(z)) - f_0(u(z))} \le M \abs{u_j(z) - u(z)},
	\end{displaymath}
	for some $M>0$. For $z\in\{0\} \times \mathbb{R}^\ell \setminus\{(0,0)\}$, using \eqref{eq:27} and \eqref{eq:28}, we get
	\begin{displaymath}
		\abs{u_j(z) - u(z)} \le C \left(M \int_{\mathbb{R}^N} g(\xi) \frac{\abs{u_j(\xi) - u(\xi)}}{d(z-\xi)^{N_\gamma -2}}\, \mathrm{d}\xi + a_j \int_{\mathbb{R}^N} \frac{g(\xi)}{d(z-\xi)^{N_\gamma-2}} \, \mathrm{d}\xi \right).
	\end{displaymath}
Since $g$ satisfies \textbf{(g1)}, from the above inequality and the previous theorem we get
\begin{equation*}
	\abs{u_j(z) - u(z)} \le C \left(M\norm{u_j - u}_\infty + a_n\right) \frac{C(g)}{d(z)^{N_\gamma -2}}.
\end{equation*}
Hence
\begin{equation}
	\label{eq:31}
	\sup_{z\in\mathbb{R}^N\setminus \{0\}}  d(z)^{N_\gamma -2} \abs{u_j(z) - u(z)} \to 0 ~\hbox{as $j\to\infty$}.
\end{equation} 

We claim that
\begin{equation}
\label{claim:3}
   \lim_{d(z)\to +\infty} d(z)^{N_\gamma -2} u(z) >0.
\end{equation}

Using \eqref{eq:27} it follows that
\begin{multline}
	\label{eq:32}
	\lim_{d(z)\to +\infty} d(z)^{N_\gamma -2} u(z) = C \lim_{d(z)\to +\infty} \int_{\mathbb{R}^N} \frac{g(\xi) f_0(u(\xi)) d(z)^{N_\gamma -2}}{d(z-\xi)^{N_\gamma -2}}\,\mathrm{d}\xi \\
	\ge C \int_{B_R} \frac{g(\xi) f_0(u(\xi)) d(z)^{N_\gamma -2}}{d(z-\xi)^{N_\gamma -2}}\,\mathrm{d}\xi
\end{multline}
for any $R>0$. For any $R>0$ there exists $z\in\{0\} \times \mathbb{R}^\ell$ such that $d(z) >2R + 1$. Hence
\begin{displaymath}
	d(z-\xi)^{N_\gamma -2} \ge \abs{d(z) - d(\xi)}^{N_\gamma -2} \ge  \abs{d(z) - R}^{N_\gamma -2} \ge 2^{2-N_\gamma} \left(1+d(z)\right)^{N_\gamma-2},
\end{displaymath}
for every $\xi\in B_R$. Using the above estimate, for $\xi\in B_R$ we get
\begin{displaymath}
	\frac{g(\xi) f_0(u(\xi)) d(z)^{N_\gamma -2}}{d(z-\xi)^{N_\gamma -2}}\le 2^{N_\gamma -2} g(\xi) f_0(u(\xi)).
\end{displaymath}
Furthermore,
\begin{displaymath}
	\frac{g(\xi) f_0(u(\xi)) d(z)^{N_\gamma -2}}{d(z-\xi)^{N_\gamma -2}}\to g(\xi) f_0(u(\xi))
\end{displaymath}
a.e. in $B_R$, as $d(z)\to \infty$. Therefore, the Dominated Convergence Theorem yields
\begin{displaymath}
	\lim_{d(z)\to +\infty} \int_{B_R} \frac{g(\xi) f(u(\xi)) d(z)^{N_\gamma -2}}{d(z-\xi)^{N_\gamma -2}} \,\mathrm{d}\xi = \int_{B_R} g(\xi) f_0(u(\xi))\,\mathrm{d}\xi.
\end{displaymath}
So, from \eqref{eq:32} we conclude that
\begin{displaymath}
	\lim_{d(z)\to +\infty} d(z)^{N_\gamma -2} u(z) \ge C \int_{B_R} g(\xi) f_0(u(\xi))\,\mathrm{d}\xi.
\end{displaymath}
Letting $R\to \infty$ and applying the Fatou's lemma we get 
\begin{displaymath}
\lim_{d(z)\to\infty} d(z)^{N_\gamma -2} u(z) \ge C \int_{\mathbb{R}^N} g(\xi) f_0(u(\xi)) \,\mathrm{d}\xi.
\end{displaymath}
So \eqref{claim:3} follows by \eqref{eq:27} and Proposition \ref{prop:positive}.

Therefore, by \eqref{eq:31} there exists $j_2\in\mathbb{N}$ and $R>>1$ such that for $j\ge j_2,\,u_j>0$ a.e. on $(\{0\} \times \mathbb{R}^\ell)\setminus B_R$. Moreover, since $u$ is continuous, there exists $\eta>0$ such that $u>\eta$ on $\overline{B_R}$. Therefore, from Proposition \ref{prop:uniform}, there exists $j_3\in\mathbb{N}$ such that for $j\ge j_3, \,u_j>0$ a.e. on $\overline{B_R} \cap (\{0\} \times \mathbb{R}^\ell)$. Thus, by choosing $j_4 = \max \left\{j_2, j_3 \right\}$, we see that for $j\ge j_4, \,u_j >0$ a.e. on $\{0\} \times \mathbb{R}^\ell$. This completes the proof.
\end{proof}

\begin{rem} \label{rem:concluding}
	The validity of the previous argument relies heavily on the representation formula \eqref{eq:27}, which is typically derived from the existence of a Green function associated to the differential operator. For linear differential operators with constant coefficients, the existence of a fundamental solution allows one to construct (continuous) solutions via convolution. by exploiting the invariance of the operator under translations. The Grushin operator $\Delta_\gamma$ is a differential operator with variable coefficients, and in particular it is invariant under translations only in the subspace $\{0\} \times \mathbb{R}^\ell$ of $\mathbb{R}^N$. So, even if a fundamental solution with singularity at a fixed point is available, it is not trivial to use it in order to derive an explicit Green function. As we have seen, it was proved in \cite{dambrosiomitidieripohozaev} that \eqref{eq:27} holds true at points of the form $(0,y)$ at which the weak solution $u$ is continuous. It is interesting to observe that the authors of \cite{dambrosiomitidieri} proved the equivalence of a representation formula like \eqref{eq:27} and of a Liouville-type result for the operator $\Delta_\gamma$, but only for $C^2$ solutions. However a full regularity theory for $\Delta_\gamma$ is not available, this approach remains unapplicable.

Furthermore, the existence of an \emph{approximate} Green function for $\Delta_\gamma$ was proved in \cite{loiudice2006} and exploited in \cite{alves2024brezis}. Unfortunately the known properties of it are too weak for our purposes, and would lead to replace (\textbf{g}$_1$) with a completely implicit assumption.
\end{rem}

\bibliographystyle{plainnat}
\bibliography{Grushin_SP_problems}{}

\bigskip

Paolo Malanchini, Dipartimento di Matematica e Applicazioni, Universit\`a degli Studi di Milano Bicocca, via R. Cozzi 55, I-20125 Milano, Italy.

\medskip 

Giovanni Molica Bisci, Dipartimento di Scienze Pure e Applicate, Universit\`a di Urbino Carlo Bo, Piazza della Repubblica 13, Urbino, Italy.

\medskip

Simone Secchi, Dipartimento di Matematica e Applicazioni, Universit\`a degli Studi di Milano Bicocca, via R. Cozzi 55, I-20125 Milano, Italy.
\end{document}